\documentclass[12pt]{amsart}
\usepackage{todonotes}
\usepackage{graphics}
\usepackage{graphicx}
\usepackage{todonotes}

\usepackage{amssymb}
\usepackage{amsmath}
\usepackage{amsthm}
\usepackage{multicol}
\usepackage[english]{babel}
\usepackage{cmap}
\usepackage{url}
\usepackage{setspace}

\usepackage{setspace}
\makeatletter
 
  \@addtoreset{equation}{section}
\makeatother
\makeatletter
\def\tbcaption{\def\@captype{table}\caption}
\makeatother
\newtheorem{theorem}{Theorem}[section]
\newtheorem{lemma}[theorem]{Lemma}

\theoremstyle{definition}
\newtheorem{definition}[theorem]{Definition}

\newtheorem*{lemma*}{Lemma}

\theoremstyle{remark}

\numberwithin{equation}{section}

\newcommand{\beq}[1]{\begin{equation}\label{#1}}
\newcommand{\eeq}{\end{equation}}

\newcommand{\bZ}{\ensuremath{\mathbb{Z}}}

\begin{document}

\title[$P_n$ as mixed $b$-concatenation of $Q_m$ and $P_k$
]{\small On Mixed $b$-concatenations of Pell and Pell-Lucas Numbers which are Pell Numbers}

\author{Kou\`essi Norbert Ad\'edji$^1$}
\thanks{$^1$ Corresponding author.}
\address{Institut de Math\'ematiques et de Sciences Physiques, Universit\'e D'Abomey-Calavi, Bénin}
\email{adedjnorb1988@gmail.com}

\author{Marija {Bliznac Trebje\v{s}anin}}
\address{University of Split, Faculty of Science, Ru\dj{}era Bo\v{s}kovi\'{c}a 33,
21000 Split, Croatia}
\email{marbli@pmfst.hr}

\begin{abstract}
Let $(P_n)_{n\ge 0}$ and $(Q_n )_{n\ge 0}$ be the Pell and Pell-Lucas sequences. Let $b$ be a positive integer such that $b\ge 2.$ In this paper, we prove that the following two Diophantine equations $P_{n}=b^{d}P_{m}+Q_{k}$ and  $P_{n}=b^{d}Q_{m}+P_{k}$ with $d,$ the number of digits of $P_k$ or $Q_k$ in base $b,$ have only finitely many solutions in nonnegative integers $(m, n, k, b, d)$. Also, we explicitly determine these solutions in cases $2 \le  b \le 10.$
\end{abstract}

\maketitle 



\section{Introduction}
Let ${P_n}$ denote the Pell sequence defined by the recurrence 
$$P_0=0,\quad P_1=1,\quad P_n=2P_{n-1}+P_{n-2},\ n\geq 2.$$
Also, let ${Q_n}$ denote the Pell-Lucas sequence defined by the recurrence 
$$Q_0=2,\quad Q_1=2,\quad Q_n=2Q_{n-1}+Q_{n-2},\ n\geq 2.$$
Many properties of these sequences are well known. Let $\alpha=1+\sqrt{2}$ and $\beta=1-\sqrt{2}$. Obviously, $2<\alpha<3$ and $-1<\beta<0$. The explicit Binet's formulas for ${P_n}$ and
${Q_n}$ are as follows
\begin{equation}
    \label{eqn:binet} P_n=\frac{\alpha^n-\beta^n}{2\sqrt{2}},\quad Q_n=\alpha^n+\beta^n,\quad \textrm{for all }n\geq 0.
\end{equation}
For $n\geq 1$, the inequalities
\begin{equation}
    \label{eq:inequalities_basic}
    \alpha^{n-2}\leq P_n\leq \alpha^{n-1}\quad \textrm{and }\quad \alpha^{n-2}\leq Q_n<\alpha^{n+1}
\end{equation}
hold. 

\begin{definition}\label{def}
Let $b\geq 2$ be an integer. Let $N$ be a positive integer  and suppose $N$ can be written as 
$$N=a_1\cdot b^d+a_2,$$
where $a_1$ and $a_2$ are positive integers and $d$ is the number of digits of $a_2$ in base $b$. Then we will call the number $N$ a $b$-concatenation of $a_1$ and $a_2$.
\end{definition}
In this paper we will consider two Diophantine equations
\begin{align}\label{eq:first}
 P_n&=b^dP_m+Q_k
\end{align}
and
\begin{align}\label{eq:second}
 P_n&=b^dQ_m+P_k,
\end{align}
and search for integer solutions $(b,n,m,k)$, such that $b\geq 2$, $n,m,k\geq 0$,  where $d$ is defined as in Definition~\ref{def}. The motivation for this work arises from the recent work of the authors of \cite{nat1}, \cite{nat}, \cite{Alan:2022}, \cite{mixed} and \cite{Banks:2005}. The main results are the following two theorems:
\begin{theorem}\label{tm:first}
Let $b\geq 2$ be an integer. Then, the Diophantine equation
$$
P_n=b^dP_m+Q_k
$$
has only finitely many solutions in nonegative integers $n,m,k$ and $d$ is the number of digits of $Q_k$ in base $b$.
Namely, we have $n<5.746\cdot 10^{27}\log^3 b.$
\end{theorem}

\begin{theorem}\label{tm:second}
Let $b\geq 2$ be an integer. Then, the Diophantine equation
$$P_n=b^dQ_m+P_k$$
has only finitely many solutions in nonegative integers $n,m,k$ and $d$ is the number of digits in $P_k$ in base $b$.
Namely, we have $n<4.7\cdot 10^{30}\log^3 b.$
\end{theorem}

\medskip

For the proofs of the Theorems~\ref{tm:first} and \ref{tm:second} we will follow the ideas of \cite{nat1} and \cite{mixed}. In fact, the proofs use Baker's theory of lower bounds for nonzero linear forms in logarithms of algebraic numbers and the reduction method due to Baker and
Davenport, the version of Bravo, G\'omez and Luca.

Some supporting results are as follows. We start with the height of an algebraic number. Let $\eta$ be an algebraic number of degree $t,$ let $a_0 \ne 0$ be the leading coefficient of its minimal polynomial over $\bZ$ and let $\eta=\eta^{(1)},\ldots,\eta^{(t)}$ denote its conjugates. The logarithmic height of $\eta$ is defined by
\[
 h(\eta)= \frac{1}{t}\left(\log |a_0|+\sum_{j=1}^{t}\log\max\left(1,\left|\eta^{(j)} \right| \right) \right).
\]
$\bullet$ If $\eta_1$ and $\eta_2$ are algebraic numbers and $j\in \bZ,$ then we have
\begin{align*}
 h(\eta_1 \pm \eta_2) &\leq h(\eta_1)+ h(\eta_2) +\log2,\\
h(\eta_1\eta_2^{\pm1}) &\leq h(\eta_1) + h(\eta_2),\\
h(\eta_1^j)&=|j|h(\eta_1).
\end{align*}
$\bullet$ If $p$ and $q$ are integers such that $q>0$ and $\gcd (p, q)=1,$ then taking $\eta=p/q$ the above definition reduces to $h(\eta)=\log(\max\{|p|,q\}).$ 

\begin{theorem}{(Theorem~9.4 of \cite{BMS:2006})}\label{tm:BMS}
Let $\gamma_1,\dots,\gamma_s$ be real algebraic numbers and let $b_1,\dots,b_s$ be nonzero integers. Let $D$ be the degree of the number field $\mathbb{Q}(\gamma_1,\dots,\gamma_s)$ over $\mathbb{Q}$ and let $A_j$ be a positive real number satisfying 
$$A_j=\max\{Dh(\gamma_j),|\log\gamma_j|,0.16\},\quad \textrm{for }j=1,\dots,s. $$
Assume that 
$$B\geq \max\{|b_1|,\dots,|b_s|\}.$$
If $\Lambda:=\gamma_1^{b_1}\cdots\gamma_s^{b_s}-1\neq 0$, then
\begin{equation}\label{ineq:BMS}
|\Lambda|\geq \exp(-1.4\cdot30^{s+3}\cdot s^{4.5}\cdot D^2(1+\log D)(1+\log B)A_1\cdots A_s).
\end{equation}
\end{theorem}

\begin{lemma}{(Lemma~7 of \cite{SGL})}\label{lemma:supporting}
If $l\geq 1$, $H>(4l^2)^l$ and $H>L/(\log L)^l$, then
$$L<2^lH(\log H)^l.$$
\end{lemma}

The following is a slight variation of the result of Dujella and Peth\H{o} (see
\cite[Lemma 5a]{Dujella-Peto}), which in turn is a generalization of the result of Baker and Davenport \cite{bd}. In fact, we use the improved version of Bravo, G\'omez and Luca (see \cite[Lemma~1]{BGF}). For a real number $x,$ we write $\left\Vert x\right\Vert$ for the distance from $x$ to the nearest integer.
\begin{lemma}\label{lemma:duj_pet}
Let $M$ be a positive integer, let $p/q$ be a convergent of the continued fraction of the irrational $\tau$ such that $q>6M$, and let $A,B,\mu$ be some real numbers with $A>0$ and $B>1$. Let 
$$\varepsilon=\|\mu q\|-M\cdot \|\tau q\|.$$
If $\varepsilon>0$, then there is no solution to the inequality
$$0<|m\tau -n+\mu|<AB^{-w},$$
in positive integers $m,n$ and $k$ with 
$$m\leq M\quad\textrm{and}\quad w\geq \frac{\log(Aq/\varepsilon)}{\log B}.$$
\end{lemma}

The following criterion of Legendre will also be useful for the proof of our claims.
\begin{lemma}{(Lemma~2.4 of \cite{AFT:2023})}\label{lemma:legendre}
Let $\kappa$ be a real number and $x,y$ integers such that 
$$\left|\kappa-\frac{x}{y} \right|<\frac{1}{2y^2}.$$
Then $x/y=p_k/q_k$ is a convergent of $\kappa$. Furthermore, let $M$ and $N$ be nonnegative integers such that $q_N>M$. Put $a(M):=\max\{a_i: i=0,1,2,\dots,N\}$. Then the inequality
$$\left|\kappa-\frac{x}{y} \right|\geq \frac{1}{(a(M)+2)y^2}$$
holds for all pairs $(x,y)$ of positive integers with $0<y<M$.
\end{lemma}

\section{The first equation $P_n=b^dP_m+Q_k$}
\subsection{Proof of Theorem~\ref{tm:first}} 

First, observe that for $m=0$ we have $P_m=0$, so the equation \eqref{eq:first} has the form
$$P_n=Q_k$$
which has only two trivial solutions $P_2=Q_0$ and $P_2=Q_1$ (it is proved in Section 4.2 in \cite{nat1}). Hence, for any base $b\geq 2$, we only have two solutions when $m=0$. From now on we assume that $m\geq 1$.
As in \cite{nat1}, from inequalities \eqref{eq:inequalities_basic} we have,
\begin{equation}\label{ineq:d_bounds}
(k-2)\frac{\log \alpha}{\log b}<d<1+(k+1)\frac{\log \alpha}{\log b}.
\end{equation}
Also, 
\begin{equation}\label{ineq:b_Q_k}
    Q_k<b^d\leq b\cdot Q_k.
\end{equation}
Now, by observing these inequalities and inserting them in \eqref{eq:first} we have
\begin{align*}
    \alpha^{n-2}\leq P_n=b^dP_m+Q_k&\leq bQ_kP_m+Q_k\leq (b+1)Q_kP_m\\
    &<(b+1)\alpha^{k+1}\alpha^{m-1}=(b+1)\alpha^{k+m}\\
    &=\alpha^{\log_{\alpha}(b+1)+k+m}.
\end{align*}
Hence, we have an upper bound for $n$ in the terms of $m$, $k$ and $b$:
\begin{equation}\label{ineq:n_upper}
    n<m+k+2+\log_{\alpha}(b+1).
\end{equation}
On the other hand, from the inequality
$$\alpha^{n-1}\geq P_n>Q_kP_m+Q_k>Q_kP_m\geq \alpha^{k+m-4}$$
we obtain a lower bound for $n$, namely
\begin{equation}\label{ineq:n_lower}
 n>m+k-3.
\end{equation}

Note that $Q_k=P_{k+1}+P_{k-1}$ for $k\geq1$. Assuming that $m\geq 1$, we can derive the following using $P_m\geq 1$ and the inequality \eqref{ineq:b_Q_k}:
\begin{align*}
    P_n&=b^dP_m+Q_k\geq b^d+Q_k>2Q_k\\
    &=2(P_{k+1}+P_{k-1})> P_{k+1}.
\end{align*}
This implies $n>k+1$, i.e. 
\begin{equation}\label{ineq:n-kvj3}
    n-k\geq 2.
\end{equation}

Now, let's rewrite equation \eqref{eq:first} by using \eqref{eqn:binet}. We have
$$\frac{\alpha^n-\beta^n}{2\sqrt{2}}=\frac{\alpha^m}{2\sqrt{2}}b^d-\frac{\beta^m}{2\sqrt{2}}b^d+Q_k.$$
After rearranging and multiplying with $2\sqrt{2}/\alpha^n$ we get
\begin{equation*}
    1-\frac{b^d}{\alpha^{n-m}}=\frac{\beta^n}{\alpha^n}-\frac{b^d\beta^m}{\alpha^n}+\frac{Q_k2\sqrt{2}}{\alpha^n}.
\end{equation*}
Since $|\beta|=1/\alpha$ we have
\begin{equation*}
    \left|1-\frac{b^d}{\alpha^{n-m}}\right|\leq \frac{1}{\alpha^{2n}}+\frac{b^d}{\alpha^{n+m}}+\frac{Q_k2\sqrt{2}}{\alpha^n}.
\end{equation*}
Now we use \eqref{eq:inequalities_basic} and \eqref{ineq:b_Q_k} and get
\begin{align*}
    \left|1-\frac{b^d}{\alpha^{n-m}}\right|&\leq \frac{1}{\alpha^{2n}}+\frac{b\cdot Q_k}{\alpha^{n+m}}+\frac{Q_k2\sqrt{2}}{\alpha^n}\\
    &< \frac{1}{\alpha^{2n}}+\frac{b\cdot \alpha^{k+1}}{\alpha^{n+m}}+\frac{\alpha^{k+1}2\sqrt{2}}{\alpha^n}\\
    &< \frac{1}{\alpha^{2n}}+\frac{b\cdot \alpha}{\alpha^{n-k}}+\frac{\alpha2\sqrt{2}}{\alpha^{n-k}}\\
    &=\frac{1+b\cdot \alpha+2\sqrt{2}\alpha}{\alpha^{n-k}}.
\end{align*}
For $b\geq 3$ we have $1+b\cdot \alpha+2\sqrt{2}\alpha<b\alpha^2$, and for $b=2$ we have $1+b\cdot \alpha+2\sqrt{2}\alpha=1+b\alpha^2$, so for $b\geq 2$ we take
$$1+b\cdot \alpha+2\sqrt{2}\alpha<1.1b\cdot \alpha^2.$$
This implies 
\begin{equation}\label{ineq:first_aps_approx_form}
\left|1-\frac{b^d}{\alpha^{n-m}}\right|<\frac{1.1b}{\alpha^{n-k-2}}.
\end{equation}
In order to apply Theorem \ref{tm:BMS}, we define
\begin{equation}\label{eq:lambda_1}
|\Lambda_1|=|1-b^d\cdot \alpha^{-(n-m)}|.
\end{equation}
Then
$$\gamma_1=b,\quad \gamma_2=\alpha,\quad b_1=d,\quad b_2=-(n-m).$$
Notice that we have $\Lambda_1\neq 0$. If we assume the contrary, that $\Lambda_1=0$, we would have $\alpha^{n-m}=b^d$. By taking norm in $\mathbb{Q}(\sqrt{2})$, we get $b^{2d}=\pm 1$ which is impossible since $b\geq 2$. 
Moreover,  it is easy to show that $D=2$, $h(\gamma_1)=\log b$ and $h(\gamma_2)=\frac{1}{2}\log \alpha$. So we take
$$A_1=2\log b,\quad A_2=\log \alpha.$$
After dividing equation \eqref{eq:first} with $P_m$ we get
$$b^d<b^d+\frac{Q_k}{P_m}=\frac{P_n}{P_m}\leq \frac{\alpha^{n-1}}{\alpha^{m-2}}=\alpha^{n-m+1}.$$
This implies
\begin{equation}\label{ineq:upper_for_d}
d<(n-m+1)\frac{\log \alpha}{\log b}\leq (n-m+1)\frac{\log \alpha}{\log 2}<1.3(n-m+1).
\end{equation}
We now have
$$B=\max\{|b_1|,|b_2|\}=\max\{d,n-m\}<1.3(n-m+1).$$
Hence, we can take $B=1.3(n-m+1)$. Also, $s=2$, so Theorem \ref{tm:BMS} implies
$$|\Lambda_1|>\exp(-1.0427\cdot 10^{10}(1+\log(1.3(n-m+1)))\log b\cdot\log \alpha).$$
Combining this with \eqref{ineq:first_aps_approx_form} we get
\begin{equation}\label{eq:first_result_BMS}
n-k-2<1.043\cdot 10^{10}\log b\cdot (1+\log(1.3(n-m+1))). 
\end{equation}
Then, we can rewrite equation \eqref{eq:first} in the form 
$$\frac{\alpha^n-\beta^n}{2\sqrt{2}}=b^dP_m+\alpha^k+\beta^k.$$
After rearranging we have 
$$\alpha^n\left(\frac{1}{2\sqrt{2}}-\alpha^{k-n}\right)-b^dP_m=\frac{\beta^n}{2\sqrt{2}}+\beta^k,$$
and again we can get an upper bound for the absolute value of the right-hand side:
$$\left|\alpha^n\left(\frac{1}{2\sqrt{2}}-\alpha^{k-n}\right)-b^dP_m\right|<\frac{|\beta^n|}{2\sqrt{2}}+|\beta^k|=\frac{1}{2\sqrt{2}\alpha^n}+\frac{1}{\alpha^k }.$$
After dividing and some manipulations we get
\begin{equation}\label{ineq:second_aps_approx_form}
    \left|1-\frac{b^dP_m}{\alpha^n\left(\frac{1}{2\sqrt{2}}-\alpha^{k-n}\right)}\right|<\dfrac{\dfrac{1}{2\sqrt{2}\alpha^n}+\dfrac{1}{\alpha^k}}{\dfrac{1}{2\sqrt{2}}-\alpha^{k-n}}\cdot \dfrac{1}{\alpha^{n}}
    <\frac{5.885}{\alpha^{n}},
\end{equation}
since $n-k\geq 2$, $n>m\geq 1$ and $k\geq 0$.
We define
$$|\Lambda_2|=\left|1-\frac{b^dP_m}{\alpha^n\left(\frac{1}{2\sqrt{2}}-\alpha^{k-n}\right)}\right|,$$
and apply Theorem \ref{tm:BMS} with parameters
$$
\gamma_1=b,\quad \gamma_2=\alpha,\quad \gamma_3=\frac{P_m}{1/2\sqrt{2}-\alpha^{k-n}},
$$
$$b_1=d,\quad b_2=-n,\quad b_3=1.$$
Note that if $|\Lambda_2|=0,$ then
\begin{align}\label{g}
\alpha^n-2\sqrt{2}\alpha^k=2\sqrt{2}b^dP_m.
\end{align}
Moreover, the Galois group of $\mathbb{Q}(\alpha)/\mathbb{Q}$ is given by
$
\{id, \sigma\},
$
where we identify the automorphism $\sigma$ to the permutation $(\alpha \beta).$ Conjugating both sides of expression \eqref{g} with $\sigma$ we obtain
$$
\beta^n+2\sqrt{2}\beta^k=-2\sqrt{2}b^dP_m.
$$
From this equality and using the fact that $n>k$ and $\beta=-\alpha^{-1},$ we get 
$$
2\sqrt{2}b^d\le |\beta^n+2\sqrt{2}\beta^k|\le \dfrac{\alpha^k+2\sqrt{2}\alpha^n}{\alpha^{n+k}}<\dfrac{1+2\sqrt{2}}{\alpha^k},
$$
which is a contradiction because $k\ge 1$ and $b\ge 2$. It follows that $\Lambda_2\ne 0. $ Again, it is easy to show that all conditions of Theorem \ref{tm:BMS} are satisfied and $D=2$, $h(\gamma_1)=\log b$, $h(\gamma_2)=\frac{1}{2}\log \alpha.$ Observe that 
\begin{align*}
h(\gamma_3)&\leq h(P_m)+h\left(\frac{1}{2\sqrt{2}}-\alpha^{k-n}\right) \\
&<(m-1)\log\alpha +\frac{1}{2}\log 8+\frac{n-k}{2}\log \alpha+\log 2. 
\end{align*}
From \eqref{ineq:n_lower} we have $m-1<n-k+2$, hence, 
$$h(\gamma_3)<\frac{3(n-k)+4}{2}\log \alpha+\frac{5}{2}\log 2.$$
We take 
\begin{align*}
    &A_1=2\log b,\quad A_2=\log \alpha,\\
    &A_3=(3(n-k)+4)\log \alpha+5\log 2.
\end{align*}
Also, $d<1.3(n-m+1)<1.3(n+1)$ so 
$$B\geq \max\{|b_1|,|b_2|,|b_3|\}=\max\{d,1,n\}<1.3(n+1)$$
and we can take $B=1.3(n+1)$. We also have $s=3$. 
With these parameters, Theorem \ref{tm:BMS} together with \eqref{ineq:second_aps_approx_form} yields
\begin{equation}\label{eq:second_result_BMS}
    n<1.94\cdot 10^{12}(1+\log(1.3(n+1)))[(3(n-k)+4)\log\alpha+5\log 2]\log b.
\end{equation}

From \eqref{eq:first_result_BMS} we have an estimate
\begin{align*}
    C&=(3(n-k)+4)\log\alpha+5\log 2=(3(n-k-2)+10)\log\alpha+5\log 2\\
    &<3.129\cdot 10^{10}(1+\log(1.3(n+1)))\log b\log \alpha+10\log\alpha+5\log 2\\
    &<2.8\cdot 10^{10}(1+\log(1.3(n+1)))\log b.
\end{align*}
Substituting this expression into \eqref{eq:second_result_BMS}, we get
$$n< 5.432\cdot 10^{22}(1+\log(1.3(n+1)))^2\log^2b.$$
 
For $n\geq 2$, we have $1+\log(1.3(n+1))<2.5\log(1.3n)$ so we can observe
$$n<3.4\cdot 10^{23}\log^2(1.3n)\log^2b.$$
We can apply Lemma \ref{lemma:supporting} to the previous inequality with parameters $L=1.3n$, $l=2$ and $H=4.42\cdot 10^{23}\log^2b$. Thus,
$$n<1.36\cdot 10^{24}\log^2b(54.45+2\log\log b)^2.$$
Since $54.45+2\log\log b<65\log^{1/2} b$ for $b\geq 2$, we have
$$n<5.746\cdot 10^{27}\log^3 b.$$
This completes the proof of Theorem \ref{tm:first}. 

\subsection{Application for $2\leq b\leq 10$}
We will now explicitly list solutions of the equation \eqref{eq:first} for some fixed $b$'s.
\begin{theorem}\label{Appl1}
Let $b$ be a positive integer such that $2\le b\le 10.$ Then the numbers $5, 12$ and $70$ are the only Pell numbers that satisfy the Diophantine equation \eqref{eq:first}. More precisely, we have the following representations:
\[
\begin{array}{ccccl}
5&=&P_3&=&3^1\cdot P_1+Q_0=3^1\cdot P_1+Q_1,\\
12&=&P_4&=&5^1\cdot P_2+Q_0=5^1\cdot P_2+Q_1,\\
12&=&P_4&=&10^1\cdot P_1+Q_0=10^1\cdot P_1+Q_1,\\
70&=&P_6&=&6^2\cdot P_1 +Q_4.
\end{array}
\]
\end{theorem}
\begin{proof}
First, we are investigating whether solutions exist for the equation \eqref{eq:first} when $m>100$.  So we assume $m\geq 101$, and we have
$96\leq m-5<n-k-2.$
Define
$$\Gamma_1:=d\log b-(n-m)\log \alpha.$$
Then from \eqref{ineq:first_aps_approx_form} we conclude
$$|\Lambda_1|=|e^{\Gamma_1}-1|<\frac{1.1b}{\alpha^{n-k-2}}<\frac{11}{\alpha^{96}}<\frac{1}{2}.$$
It is easy to see if $y<1/2$, then $|e^x-1|<y$ implies $|x|<2y$, hence
$$|\Gamma_1|<\frac{2.2b}{\alpha^{n-k-2}}.$$
After dividing with $ \log \alpha$, we observe
\begin{equation}\label{ineq:final_tm_1}
0<\left|d\frac{\log b}{\log \alpha}-(n-m)\right|<\frac{2.2b}{\log\alpha \cdot \alpha^{n-k-2}}.
\end{equation}
Using the fact that $b\leq 10$ and the upper bound for $n$ from Theorem \ref{tm:first} on \eqref{ineq:upper_for_d} we obtain an upper bound for $d$, namely
\begin{equation}\label{ineq:upper_d_numeric}
d<1.3(n-m+1)<1.3\cdot(5.746\cdot 10^{27}\log^3 10-100)<9.12\cdot 10^{28}.
\end{equation}
After dividing \eqref{ineq:final_tm_1} with $d$ we observe
\begin{equation}\label{ineq:final_tm_2}
\left|\frac{\log b}{\log \alpha}-\frac{n-m}{d}\right|<\frac{2.2b}{\log\alpha \cdot \alpha^{n-k-2}\cdot d}.
\end{equation}
Since $n-k-2> 96$, 
$$\frac{\log\alpha \cdot \alpha^{n-k-2}}{4.4b}>\frac{\log\alpha \cdot \alpha^{96}}{44}>1.11\cdot 10^{35}>d.$$
Hence,
$$\left|\frac{\log b}{\log \alpha}-\frac{n-m}{d}\right|<\frac{1}{2d^2}.$$

We see that we can apply Lemma \ref{lemma:legendre} for $\kappa=\frac{\log b}{\log \alpha}$ and $y=d$, $M=9.12\cdot 10^{28}.$
Let $q_t$ be the denominator of the $t$-th convergent of the continued fraction of $\kappa$. The application of Lemma \ref{lemma:legendre} yields the following data.
\begin{table}[h]
 \resizebox{\textwidth}{!}{%
\begin{tabular}{|c|c|c|c|c|c|c|c|c|c|c|}
    \hline
    $b$ & $2$& $3$& $4$&$5$&$6$&$7$&$8$&$9$&$10$  \\
    \hline
    $q_t$& $q_{56}$& $q_{59}$& $q_{66}$& $q_{55}$& $q_{43}$& $q_{58}$& $q_{56}$& $q_{53}$& $q_{67}$\\
    \hline
    $q_t>$& $10^{29}$&$ 10^{30}$&$2\cdot 10^{29}$&$ 10^{29}$&$6\cdot 10^{29}$&$ 10^{29}$&$10^{29}$&$7\cdot 10^{29}$&$3\cdot 10^{29}$\\
    \hline
    $a(M)$&$a_{28}$&$a_{27}$&$a_{59}$&$a_{17}$&$a_{9}$&$a_{8}$&$a_{25}$&$a_{5}$&$a_{24}$\\
    \hline
    $a(M)$&$100$&$130$&$110$&$163$&$509$&$33$&$34$&$68$&$52$\\
    \hline
\end{tabular}
}
\end{table}

In all these cases it holds
$$\left|\frac{\log b}{\log \alpha}-\frac{n-m}{d}\right|\geq \frac{1}{(509+2)d^2}.$$
When combined with \eqref{ineq:final_tm_2}, we obtain
$$d>\frac{\log \alpha\cdot \alpha^{n-k-2}}{2.2b\cdot 511}> \frac{\log \alpha\cdot \alpha^{96}}{22\cdot 511}>4.37\cdot 10^{32},$$
which is a contradiction with \eqref{ineq:upper_d_numeric}.

Now, it remains to observe $m\leq 100$.
From \eqref{ineq:n_upper} we have
$$n-k<m+2+\frac{\log(b+1)}{\log \alpha}\leq 102+\frac{\log 11}{\log \alpha}<104.7,$$
or more precisely, $n-k\leq 104.$
We substitute this upper bound for $n-k$ into the inequality \eqref{eq:second_result_BMS} and obtain
$$n< 1.268\cdot 10^{15}(1+\log(1.3(n+1))),$$
which implies $n< 5.037\cdot 10^{16}$.

Let's assume that $n\geq 3$. Define 
$$\Gamma_2=d\log b-n\log\alpha+\log\frac{P_m}{1/2\sqrt{2}-\alpha^{k-n}}.$$
Then from \eqref{ineq:second_aps_approx_form} we have
$$|\Lambda_2|=|e^{\Gamma_2}-1|<\frac{5.885}{\alpha^n}<\frac{1}{2}.$$
As before, we conclude
$$|\Gamma_2|<\frac{11.77}{\alpha^{n}}.$$ After dividing this inequality with $\log \alpha$, we observe
$$0<\left|d\frac{\log b}{\log \alpha}-n+\frac{\log\frac{P_m}{1/2\sqrt{2}-\alpha^{k-n}}}{\log\alpha}\right|<\frac{11.77/\log \alpha}{\alpha^{n}}.$$
We can now apply Lemma \ref{lemma:duj_pet} with parameters
\begin{align*}
    &w:=n,\quad A:=11.77/\log\alpha,\quad B:=\alpha,\\ & M:=6.55\cdot 10^{16}>1.3n\geq 1.3(n-m+1)>d,\\
    &\tau:=\frac{\log b}{\log \alpha},\quad \mu:=\frac{\log(P_m/(1/2\sqrt{2}-\alpha^{k-n}))}{\log\alpha}.
\end{align*}
Let's denote with $q_t$ the denominator of the $t$-th convergent of the continued fraction of $\tau$. We implemented the algorithm of Lemma \ref{lemma:duj_pet} in Wolfram Mathematica and for $2\leq b\leq 10$, $1\leq m\leq 100$ and $3\leq n-k\leq 104$ we got the largest numerical bounds for $w=n$ for the following data.
\begin{table}[h!]
 \resizebox{\textwidth}{!}{%
\begin{tabular}{|c|c|c|c|c|c|c|c|c|c|c|}
    \hline
    $b$ & $2$& $3$& $4$&$5$&$6$&$7$&$8$&$9$&$10$  \\
    \hline 
    $m$ & $27$&$26$&$15$&$19$&$22$&$1$&$16$&$27$&$19$\\
    \hline
    $n-k$& $58$&$53$&$46$&$28$&$58$&$41$&$47$&$55$&$33$\\
    \hline
    $q_t$& $q_{37}$& $q_{33}$& $q_{43}$& $q_{34}$& $q_{27}$& $q_{35}$& $q_{39}$& $q_{31}$& $q_{43}$\\
    \hline
    $\epsilon>$& $ 8\cdot 10^{-6}$&$4\cdot 10^{-5}$&$0.002$&$ 10^{-4}$&$10^{-4}$&$ 10^{-4}$&$ 2\cdot10^{-4}$&$6\cdot 10^{-4}$&$10^{-4}$\\
    \hline
    $n\leq $&$64$&$62$&$57$&$61$&$55$&$61$&$59$&$58$&$59$\\
    \hline
\end{tabular}
}
\end{table}

Notice that we have $n\leq 64$ in any case. To investigate these remaining cases, we have written a short computer program to search the parameters
$n, m$ and $k$ that satisfy \eqref{eq:first}. We have set the ranges as follows: $0\leq k\leq n-2\leq 62$, $1\le m\le 100$, and $m+1\leq n\leq 64$. The program only yields solutions as described in Theorem~\ref{Appl1}.
\end{proof}

\section{The second equation $P_n=b^dQ_m+P_k$}

\subsection{Proof of Theorem~\ref{tm:second}}

We start by establishing some useful inequalities. First, we have from \eqref{eq:inequalities_basic} that
\[ 
d = \lfloor\log_{b}P_{k}\rfloor+1 \le 1 + \log_{b}P_{k} < 1+ \log_{b}(\alpha^{k-1})= 1+ (k-1)\frac{\log\alpha}{\log b}
\]
and
\[ 
d = \lfloor\log_{b}P_{k}\rfloor+1 >  \log_{b}P_{k} \geq  \log_{b}(\alpha^{k-2})= (k-2)\frac{\log\alpha}{\log b}.
\]
We deduce that
\begin{equation}\label{Eq11}
(k-2)\frac{\log\alpha}{\log b}< d < 1+ (k-1)\frac{\log\alpha}{\log b}.
\end{equation}
Second, we have
\[
P_{k}=b^{\log_{b}P_{k}}< b^{d} \leq b^{1 + \log_{b}P_{k}}=b\cdot b^{\log_{b}P_{k}}=b\cdot P_{k},
\]
which leads to
\begin{equation} \label{Eq2}
P_{k}<b^{d}\le b\cdot P_{k}.
\end{equation}
By \eqref{Eq2}, we can easily see that
\[
\alpha^{n-2}\le P_{n}=b^{d}Q_{m}+P_{k}\leq b\cdot P_{k}\cdot Q_{m}+P_{k}<(b+1)\cdot Q_{m}\cdot P_{k+1}.
\]
Moreover, we have 
\[
(b+1)\cdot Q_{m}\cdot P_{k+1}< (b+1)\cdot \alpha^{m+1}\cdot\alpha^k=(b+1)\alpha^{m+k+1}=\alpha^{\log_{\alpha}(b+1)}\cdot\alpha^{m+k+1}
\]
and then
\begin{align}\label{n_1}
\alpha^{n-2}<\alpha^{m+k+1+\log_{\alpha}(b+1)}.
\end{align}
Also, we get
\[
\alpha^{n-1}\ge P_{n}=b^{d}Q_{m}+P_{k}> P_{k}\cdot Q_{m}+P_{k}> P_{k}\cdot Q_{m}\ge \alpha^{m+k-4},
\]
which becomes
\begin{align}\label{n_2}
\alpha^{n-1}>\alpha^{m+k-4}.
\end{align}
Combining inequalities \eqref{n_1} and \eqref{n_2}, we obtain
\begin{equation}\label{Equ3}
 m+k-3< n < m + k +3+\dfrac{\log (b+1)}{\log\alpha}.
\end{equation}
Since $m\ge 0$ we can deduce from \eqref{Eq2} that 
$$
P_n=b^dQ_m+P_k>2b^d+P_k>3P_k
$$
and therefore 
\begin{align}
n-k\ge 1,
\end{align}
where you use inequalities \eqref{eq:inequalities_basic}. From \eqref{eqn:binet} we rewrite the Diophantine equation \eqref{eq:second}
as
\[ 
\frac{\alpha^{n}-\beta^{n}}{2\sqrt{2}} = \left( \alpha^{m} + \beta^{m} \right)b^{d}+P_{k},
\]
which leads to
$$
\frac{\alpha^{n}}{2\sqrt{2}}-\alpha^{m}b^{d} = \frac{\beta^{n}}{2\sqrt{2}}+\beta^{m} b^{d} + P_{k}
$$
and therefore we have
$$
1-\dfrac{2\sqrt{2}\cdot b^d}{\alpha^{n-m}}= \frac{\beta^{n}}{\alpha^n}+\dfrac{2\sqrt{2}\cdot\beta^{m}\cdot b^{d}}{\alpha^n} + \dfrac{2\sqrt{2}\cdot P_{k}}{\alpha^n}
$$
Taking absolute values of both sides of the above equality, we get that
$$
\left|1-\dfrac{2\sqrt{2}\cdot b^d}{\alpha^{n-m}}\right|< \dfrac{1}{\alpha^{2n}}+\dfrac{2\sqrt{2}\cdot b^d}{\alpha^{n+m}}+\dfrac{2\sqrt{2}\cdot P_k}{\alpha^n}.
$$
Since $P_{k}<b^{d}<b\cdot P_{k}$ and using \eqref{eq:inequalities_basic} with $2\sqrt{2}<\alpha^2,$  we obtain the following estimates
\begin{align*}
\left| 1-\frac{2\sqrt{2}\cdot b^d}{\alpha^{n-m}}\right| & <  \frac{1}{\alpha^{2n}}+\frac{b\cdot\alpha^{k-1}\cdot 2\sqrt{2}}{\alpha^{n+m}} + \frac{2\sqrt{2}\cdot \alpha^{k-1}}{\alpha^{n}}\\
 & <  \frac{1}{\alpha^{2n}}+\frac{\alpha\cdot b}{\alpha^{n+m-k}} + \frac{\alpha}{\alpha^{n-k}}\\
 & <  \frac{1}{\alpha^{n-k}}+\frac{\alpha\cdot b}{\alpha^{n-k}} + \frac{\alpha}{\alpha^{n-k}}=\dfrac{1+\alpha\cdot b+\alpha}{\alpha^{n-k}}.
\end{align*}
Note that for $b\ge 2,$ we have $1+\alpha\cdot b+\alpha<2\alpha\cdot b.$
Thus, we obtain 
\begin{equation}\label{upper_La}
|\Lambda_3|:=\left| 1 - \frac{2\sqrt{2}\cdot b^{d}}{\alpha^{n-m}}\right| < \frac{2b}{\alpha^{n-k-1}}.
\end{equation}
To apply Theorem~\ref{tm:BMS}, we take
$$
\begin{array}{lcll}
\gamma_{1}=b, & \gamma_{2}=2\sqrt{2}, & \gamma_{3}=\alpha,\\
b_{1}=d, & b_{2}=1, & b_{3}=-(n-m).
\end{array}
$$ 
Assuming that $\Lambda_3=0$, we get 
\begin{align}\label{nonzero}
\alpha^{n-m}=2\sqrt{2}\cdot b^{d}.
\end{align}
 Taking the norm in $\mathbb{Q}(\sqrt{2})$ of both sides of \eqref{nonzero}, we get
$\pm 1= 8b^{2d}$,  which is impossible. Then $\Lambda_3 \ne 0$.
We know that $\gamma_{1}, \gamma_{2}, \gamma_{3}$ are elements of  $\mathbb{L}=\mathbb{Q}(\sqrt{2})$ and $D:=[\mathbb{L}: \mathbb{Q}]=2.$ Using the properties of the function $h(\cdot)$, we get
$$ 
h(\gamma_{1})=h(b)=\log b, \quad h(\gamma_{2})=h(2\sqrt{2})=h(\sqrt{8})=\frac{1}{2}\log8
$$
and  
$$
h(\gamma_{3})=h(\alpha)=\frac{1}{2}\log\alpha.
$$ 
With the notations from Theorem~\ref{tm:BMS} we can take the following parameters
$$
A_{1}=2\log b,\quad A_{2}=\log8,\quad \text{and}\quad A_{3}=\log\alpha.
$$
By 
$P_{n}=b^{d}Q_{m}+P_{k}$ and dividing its both sides  by $Q_m,$ we get 
$$
b^d<  b^{d}+ \frac{P_{k}}{Q_{m}}=\frac{P_{n}}{Q_{m}}\le \alpha^{n-m+1},
$$
and therefore 
$$
d <(n-m+1)\cdot\frac{\log\alpha}{\log b}\le (n-m+1)\frac{\log\alpha}{\log 2}.
$$
Because $b\ge 2,$ we get
\begin{align}\label{bound-d}
d <  1.3\cdot (n-m+1).
   \end{align}
Using the fact that
$$
\max\lbrace \vert b_{1}\vert, \vert b_{1}\vert, \vert b_{3}\vert\rbrace = \max\lbrace d; 1; n-m\rbrace<1.3\cdot (n-m+1),
$$ 
we can take $B=1.3\cdot (n-m+1)$. Since in this case three algebraic numbers are considered, then $s=3.$ Applying Theorem~\ref{tm:BMS} to $\Lambda_3$ leads to
\begin{align*}
\log |\Lambda_3|&>-1.4\cdot 30^6\cdot 3^{4.5}\cdot 2^{2}\cdot (1+\log 2)\cdot (1+\log (1.3(n-m+1)))\\
&\times 2\log b\cdot \log 8\cdot \log\alpha.
\end{align*}
Combining this with \eqref{upper_La}, we see that
\begin{equation}\label{Equ6}
 n-k-1<4.1\cdot 10^{12}\cdot\log b\cdot (1+\log (1.3(n-m+1))),
\end{equation}
which is valid for $b\ge 2.$ The next step is to rewrite the Diophantine equation \eqref{eq:second} in the form 
$$
\frac{\alpha^{n}-\beta^{n}}{2\sqrt{2}}=b^{d}Q_{m} + \frac{\alpha^{k}-\beta^{k}}{2\sqrt{2}},
$$
which also leads to
$$
\dfrac{\alpha^n}{2\sqrt{2}}\cdot\left(1 - \alpha^{k-n}\right)- b^{d}Q_{m} = \frac{\beta^{n}}{2\sqrt{2}} -\frac{\beta^{k}}{2\sqrt{2}}.
$$
Taking absolute value of both sides of the above equality and dividing by $\dfrac{\alpha^n}{2\sqrt{2}}\left(1-\alpha^{k-n} \right)$, we get 
$$
\left| 1 - \frac{2\sqrt{2}\cdot b^{d}\cdot Q_{m}}{\alpha^{n}\left(1 - \alpha^{k-n}\right)} \right|< \dfrac{1}{1- \alpha^{k-n}}\cdot \left(\dfrac{|\beta|^n}{\alpha^n}+\dfrac{|\beta|^k}{\alpha^n} \right).
$$
Moreover, we have
\begin{align*}
\left| 1 - \frac{2\sqrt{2}\cdot b^{d}\cdot Q_{m}}{\alpha^{n}\left(1- \alpha^{k-n}\right)} \right| & <  \frac{1}{1 - \alpha^{k-n}}\cdot \left(\frac{1}{\alpha^{2n}} + \frac{1}{\alpha^{n+k}}\right)\\ 
 & =  \frac{\alpha^{n-k}}{\alpha^{n-k}-1}\cdot \left(\frac{1}{\alpha^{2n}} + \frac{1}{\alpha^{n+k}}\right)\\
 & <  \frac{\alpha^{n-k}}{\alpha^{n-k}-1}\cdot \left( \frac{2}{\alpha^{n}} \right).
\end{align*} 
The above inequalities hold because for $n-k\ge 1,$ we have  $n> k.$ Note also that if $n-k\ge 1,$ then
$$
0<\frac{\alpha^{n-k}}{\alpha^{n-k}-1}<1.71,
$$
and therefore
\begin{align}\label{upper_La2}
|\Lambda_4|:=\left| 1 - \frac{2\sqrt{2}\cdot b^{d}\cdot Q_{m}}{\alpha^{n}\left(1- \alpha^{k-n}\right)}  \right| <  \frac{3.42}{\alpha^{n}}.
\end{align}
Thus, we have to apply Theorem~\ref{tm:BMS} to \eqref{upper_La2} by taking the following data
$$
\begin{array}{lcll}
\gamma_{1}=b, & \gamma_{2}=\alpha, & \gamma_{3}=\dfrac{2\sqrt{2}\cdot Q_m}{1-\alpha^{k-n}},\\
b_{1}=d, & b_{2}=-n, & b_{3}=1.
\end{array}
$$
Note that $\Lambda_4\ne 0.$ To see this we can follow the same idea for the proof of $\Lambda_2\neq 0.$ The heights of the algebraic numbers $\gamma_1, \gamma_2$ and $\gamma_3$ are defined respectively by 
$$
h(\gamma_{1})=\log b, \quad h(\gamma_{2})=\frac{1}{2}\log \alpha 
$$ 
and  
\begin{align*}
h(\gamma_{3}) & =  h\left(\frac{2\sqrt{2}\cdot Q_{m}}{1 - \alpha^{k-n}}\right)  \le h (2\sqrt{2})+  h(Q_{m}) + h\left(1 - \alpha^{k-n}\right)\\
 &  \le  \log Q_{m} + h (2\sqrt{2}) + h(\alpha^{k-n}) + \log 2\\
 &  <  (m+1)\log\alpha + \frac{1}{2}\log 8 + \frac{n-k}{2}\log\alpha +\log 2\\
 &=(m+1)\log\alpha + \frac{5}{2}\log 2 + \frac{n-k}{2}\log\alpha.
 \end{align*}
Using inequalities \eqref{Equ3}, we get that  $m+1< (n-k)+4.$ It follows that  
$$
h(\gamma_{3})< \frac{3(n- k) + 8}{2}\log \alpha + \frac{5}{2}\log 2.
$$
Then, we can take 
$$
A_1=2\log b,\quad A_3=\log\alpha\quad \text{and}\quad
A_{3}= (3(n-k)+8)\log \alpha + 5 \log 2.
$$ 
As $B\ge  \max\lbrace \vert b_{i}\vert\rbrace = \max\lbrace d, 1, n\rbrace $ and $d<1.3\cdot (n-m+1)<1.3(n+1),$ then we can take $B=1.3(n+1).$
Also, in this case $s=3.$  Thus, combining \eqref{upper_La2} with Theorem~\ref{tm:BMS} we see that
\begin{align*}
n\log\alpha-\log 3.42 &<1.4\cdot 30^{6}\cdot 3^{4.5}\cdot 2^{2}\cdot (1+\log 2)\cdot (1+\log (1.3(n+1))) \\
&\times  2\log b\cdot \log\alpha\cdot\left( (3(n-k)+8)\log \alpha + 5 \log 2\right),
\end{align*}
which can be reduce to
\begin{align}\label{EEqu6}
n < 2\cdot 10^{12}\cdot (1+\log (1.3(n+1)))\cdot\left[(3(n-k)+8)\log\alpha + 5 \log 2\right]\cdot\log b.
\end{align}
So, from inequality \eqref{Equ6}, we get that
\begin{align*}
C & =  (3(n-k)+8)\log\alpha + 5 \log 2\\
 & =  3(n-k)\log\alpha + 8\log\alpha + 5 \log 2\\
  & <  4.1\cdot 10^{12}\cdot 3\log\alpha\cdot (1+\log (1.3(n+1)))\cdot\log b + 11\log\alpha  + 5 \log 2\\
  & <  1.2\cdot 10^{13}\cdot (1+\log (1.3(n+1)))\cdot\log b.
 \end{align*}
Inserting this in \eqref{EEqu6}, we get
$$ 
n< 2.4\cdot 10^{25}\cdot (1+\log (1.3(n+1)))^{2} \cdot\log^{2} b.
$$
Moreover, we have $1+\log (1.3(n+1))<3\log (n+1)$ which is valid for all $n\ge 1.$ 
It follows that 
$$
n+1<2.2\cdot 10^{26}\cdot\log^2 (n+1)\cdot \log^2 b.
$$
So, to get an upper bound of $n$ in terms of $b,$ we need to refer to Lemma~\ref{lemma:supporting} by putting 
$$
l=2,\quad L=n+1\quad\text{ and}\quad H=2.2\cdot 10^{26}\cdot\log^{2} b.
$$ 
Then, Lemma~\ref{lemma:supporting} tell us that 
$$
n+1<2^2\cdot 2.2\cdot 10^{26}\cdot\log^{2} b\cdot (60.7+2\log\log b)^2 .
$$
Since for $b\ge 2,$ we have $60.7+2\log\log b<73\log^{1/2} b,$ then 
\begin{align}\label{nn2}
n<4.7\cdot 10^{30}\cdot \log^3 b.
\end{align}
 This completes the proof of Theorem~\ref{tm:second}.

\subsection{Application for $2\le b\le 10$}

Our result in this case is the following.

\begin{theorem}\label{Appl2}
Let $b$ be a positive integer such that $2\le b\le 10.$ Then, the numbers $5, 12, 29$ and $70$ are the only Pell numbers that satisfy the Diophantine equation \eqref{eq:second}. More precisely, we have the following representations:
\[
\begin{array}{ccccl}
5&=&P_3&=&2^1\cdot Q_0+P_1=2^1\cdot Q_1+P_1,\\
12&=&P_4&=&5^1\cdot Q_0+P_2=5^1\cdot Q_1+P_2=2^1\cdot Q_2+P_0,\\
12&=&P_4&=&6^1\cdot Q_0+P_0=6^1\cdot Q_1+P_0,\\
29&=&P_5&=&2^1\cdot Q_3+P_1,\\
70&=&P_6&=&5^1\cdot Q_3+P_0.
\end{array}
\]
\end{theorem}
\begin{proof}
First, let us prove unconditionally that if equation \eqref{eq:second} is valid then $m\le  100.$ For that we assume the opposite i.e. $m > 100.$ Put
\begin{align}\label{Ga2}
\Gamma_3:=d\log b-(n-m)\log\alpha+\log (2\sqrt{2}).
\end{align}
Because $96<m-4<n-k-1$ and $2\le b\le 10,$ we deduce from \eqref{upper_La}
$$
|\Lambda_3|:=|1-e^{\Gamma_3}|<\dfrac{2b}{\alpha^{n-k-1}}<\dfrac{1}{2}.
$$
It follows that
$$
|\Gamma_3|<\dfrac{4b}{\alpha^{n-k-1}}.
$$
Combining this with \eqref{Ga2}, we have that
\begin{align}\label{Re2}
0<\left|d\dfrac{\log b}{\log\alpha}-(n-m)+\dfrac{\log (2\sqrt{2})}{\log\alpha} \right|<\dfrac{4b/\log\alpha}{\alpha^{n-k-1}}.
\end{align}
Since $2\le b\le 10,$ we have $n<4.7\cdot 10^{30}\log^3 b<5.74\cdot 10^{31}$ and then
$$
d<1.3(n-m+1)<1.3(n+1)<7.5\cdot 10^{31}.
$$ 
Thus, we can apply Lemma~\ref{lemma:duj_pet} to \eqref{Re2} with the following data: 
$$M:=7.5\cdot 10^{31},\ A:=4b/\log\alpha,\ B:=\alpha,\ w:=n-k-1,$$
$$
\tau:=\dfrac{\log b}{\log\alpha}\quad \text{and}\quad \mu:=\dfrac{\log (2\sqrt{2})}{\log\alpha}.
$$
Let $q_t$ be the denominator of the $t$-th convergent of the continued fraction of $\tau.$ We apply Lemma~\ref{lemma:duj_pet} with Wolfram Mathematica and we got the following results:
\begin{center}
	\begin{tabular}{|c|c|c|c|c|c|c|c|c|c|}
		\hline $b$ & $ 2 $ & $ 3 $ & $ 4 $ & $ 5 $ & $ 6 $ & $ 7 $ & $ 8 $ & $ 9 $ & $10$ \\ \hline $q_{t}$ & $q_{65}$ &  $q_{64}$ & $q_{72}$ & $q_{63}$ & $q_{53}$ & $q_{62}$ & $q_{63}$ & $q_{58}$ & $q_{73}$     \\ \hline
		$n-k-1< $ & $89$ & $90 $ & $90 $ & $90 $ & $93 $ & $90$ & $91 $ & $91 $ & $92$\\ \hline 
	$\epsilon> $ & $0.45$ & $0.24 $ & $0.45 $ & $0.40$ & $0.43 $ & $0.45$ & $0.49 $ & $0.24 $ & $0.17$\\ \hline 	
	\end{tabular}\\	
\end{center} 
In all cases we conclude according to the values of $b$ that $n-k-1< 93,$  which contradicts the fact that $96<m-4<n-k-1$. Therefore, we have $m\le 100.$ Secondly, by \eqref{Equ3} we get that
$$
n-k\le m+3+\dfrac{\log (b+1)}{\log\alpha}\le 106.
$$
Substituting the upper bound for $n -k$ into \eqref{EEqu6} and using the fact that $2\le b\le 10$, we see that
$$
n< 1.34\cdot 10^{15} \cdot (1+\log (1.3(n+1))),
$$
which also implies that $n<5.33\cdot 10^{16}.$
Let 
\begin{align}
\Gamma_4:=d\log b-n\log\alpha+\log \left(\dfrac{2\sqrt{2}\cdot Q_m}{1-\alpha^{k-n}} \right)
\end{align}
so that
$$
|\Lambda_4|:=|e^{\Gamma_4}-1|<\dfrac{3.42}{\alpha^{n}}.
$$
Note that for $n\ge 3,$ we have $3.42/\alpha^{n}<1/2$ which also implies that 
\begin{align}\label{Gamm2}
|\Gamma_4|<\dfrac{6.84}{\alpha^{n}}.
\end{align}
Therefore, dividing both sides of \eqref{Gamm2} by $\log\alpha,$ we get that 
\begin{align*}
0<\left|d\dfrac{\log b}{\log\alpha}-n+\dfrac{\log \left(2\sqrt{2}\cdot Q_m/(1-\alpha^{k-n}) \right)}{\log\alpha} \right|<\dfrac{6.84/\log\alpha}{\alpha^{n}}.
\end{align*}
Next we take the following parameters $$w:=n,\ A:=6.84/\log\alpha,\ B:=\alpha,\ M:=7\cdot 10^{16}>1.3(n+1)>d,$$
$$
\tau:=\dfrac{\log b}{\log\alpha}\quad \text{and}\quad \mu:=\dfrac{\log \left(2\sqrt{2}\cdot Q_m/(1-\alpha^{k-n}) \right)}{\log\alpha}
$$
in view to apply Lemma~\ref{lemma:duj_pet}. Let $q_t$ be the denominator of the $t$-th convergent of the continued fraction of $\tau.$ Now, using Wolfram Mathematica we apply Lemma~\ref{lemma:duj_pet} and for $2\le b\le 10$, $0\le m\le 100$ and $1\le n-k\le 106$ we got the largest numerical bounds for $w=n$ for the following data.
\begin{table}[h!]
\resizebox{\textwidth}{!}{%
	\begin{tabular}{|c|c|c|c|c|c|c|c|c|c|}
		\hline $b$ & $ 2 $ & $ 3 $ & $ 4 $ & $ 5 $ & $ 6 $ & $ 7 $ & $ 8 $ & $ 9 $ & $10$ \\ \hline $q_{t}$ & $q_{37}$ &  $q_{33}$ & $q_{43}$ & $q_{34}$ & $q_{27}$ & $q_{35}$ & $q_{39}$ & $q_{31}$ & $q_{43}$     \\ \hline
		$ m$& $28$ & $20$& $7$& $30$& $12$& $27$& $3$& $18$& $8$ \\ \hline
		$ n-k$& $54$& $48$& $53$& $18$& $41$& $38$& $9$& $38$& $3$\\ \hline
		$n < $ & $61$ & $60 $ & $60$ & $60 $ & $60 $ & $58 $ & $57 $ & $59 $ & $60$\\ \hline 
	$\epsilon> $ & $10^{-4}$ & $10^{-3} $ & $9\cdot 10^{-5} $ & $2\cdot 10^{-4} $ & $9\cdot 10^{-5} $ & $10^{-4} $ & $7\cdot 10^{-4} $ & $2\cdot 10^{-4} $ & $3\cdot 10^{-5}$\\ \hline 	
	\end{tabular}
 }
\end{table} 

 Thus we can deduce that $n<61$ holds in all cases. Next, using \eqref{Equ3} we get $m+k<64.$ Finally, we have written a short computer program to search the parameters
$d, n, m$ and $k$ that satisfy \eqref{eq:second} in the range $0\le m\le 100$ and we found only the Pell numbers given in the statement of Theorem~\ref{Appl2}. This completes the proof. 
\end{proof}

\section*{Acknowledgements}
The authors are grateful to the referee for the useful remarks and suggestions to improve
the quality of this paper. The first  author is supported by IMSP, Institut de Math\'ematiques et de Sciences Physiques de l'Universit\'e d'Abomey Calavi.

\end{document}